\documentclass[10pt, final]{article}
\usepackage{amsthm}
\usepackage{a4}
\usepackage{amsmath}%
\usepackage{amstext}%
\usepackage{amssymb}%
\usepackage{showkeys}%
\usepackage{epsfig}%
\usepackage{cite}%
\DeclareMathOperator*\inte{int}%
\DeclareMathOperator*\ri{ri}%
\DeclareMathOperator*\rint{rint}%
\DeclareMathOperator*\cl{cl}%
\DeclareMathOperator*\ic{ic}%
\DeclareMathOperator*\epi{epi}%
\DeclareMathOperator*\dom{dom}%
\DeclareMathOperator*\aff{aff}
\DeclareMathOperator*\cone{cone}%
\DeclareMathOperator*\core{core}%
\DeclareMathOperator*\icr{icr}%
\DeclareMathOperator*\co{co}%
\DeclareMathOperator*\id{id}%
\DeclareMathOperator*\im{Im}%
\DeclareMathOperator*\pr{pr}%

\newtheorem{thm}{Theorem}[section]
\newtheorem{lem}{Lemma}[section]
\newtheorem{prop}{Proposition}[section]
\newtheorem{cor}{Corollary}[section]
\newtheorem{defn}{Definition}[section]
\newtheorem{ex}{Example}[section]
\newtheorem{rmk}{Remark}[section]

\newcommand{\Real}{{\mathbb R}}
\newcommand{\oReal}{{\overline{\mathbb R}}}

\newcommand{\To}{\longrightarrow}
\def\1{\^{\i}}
\def\2{\u{a}}
\def\3{\c{s}}
\def\4{\^{a}}
\def\5{\c{t}}
\def\a{\alpha}
\def\b{\beta}

\def\d{\delta}

\def\l{\lambda}
\def\<{\langle}
\def\>{\rangle}

\title{On the generalized parallel sum of two maximal monotone operators of Gossez type (D)}

\author{Radu Ioan Bo\c{t} \thanks
{Faculty of Mathematics, Chemnitz University of Technology,
D-09107 Chemnitz, Germany, e-mail: radu.bot@mathematik.tu-chemnitz.de. Research partially supported by DFG (German Research Foundation), project WA 922/1-3.} \and Szil\'ard L\'aszl\'o \thanks {Faculty of Mathematics and Computer Science, Babe\c{s}-Bolyai University, Cluj-Napoca, Romania, e-mail: laszlosziszi@yahoo.com. Research done during the stay of the author in the academic year 2010/2011 at Chemnitz University of Technology as a guest of the Chair of Applied Mathematics (Approximation Theory). The author wishes to thank for the financial support provided from programs co-financed by The Sectoral Operational Programme Human Resources Development, Contract POSDRU 6/1.5/S/3 –- ``Doctoral studies: through science towards society''.}}

\begin{document}
\maketitle

\noindent \textbf{Abstract.} The generalized parallel sum of two monotone operators via a linear continuous mapping is defined as the inverse of the sum of the inverse of one of the operators and with inverse of the composition of the second one with the linear continuous mapping. In this article, by assuming that the operators are maximal monotone of Gossez type (D), we provide sufficient conditions of both interiority- and closedness-type for guaranteeing that their generalized sum via a linear continuous mapping is maximal monotone of Gossez type (D), too. This result will follow as a particular instance of a more general one concerning the maximal monotonicity of Gossez type (D) of an extended parallel sum defined for the maximal monotone extensions of the two operators to the corresponding biduals.\\



\section{Introduction and preliminaries}

Having two nonempty sets $A$ and $B$ and a multivalued operator $M :A\rightrightarrows B$, we denote by $G(M) =\{(a,b) \in A \times B : b \in M(a)\}$ its \textit{graph} and by $M^{-1} : B\rightrightarrows A$ the \textit{inverse operator} of $M$, which is the multivalued operator having as graph the set $G(M^{-1}) := \{(b,a) \in B \times A : (a,b) \in G(M)\}$. When $X$ is a real nonzero Banach space and $X^*$ its topological dual space,  the \textit{parallel sum} of two multivalued monotone operators $S,T:X\rightrightarrows X^*$ is defined as
$$S||T : X\rightrightarrows X^*,\ S||T(x) :=(S^{-1}+T^{-1})^{-1}(x) \ \forall x \in X.$$
This notion has been first considered in Hilbert spaces by Passty in \cite{Pa}, where the interested reader can find some practical interpretations of this notion including some preliminary investigations on the maximal monotonicity of the parallel sum of two maximal monotone operators. The latter problem was also addressed in Hilbert spaces in \cite{Mou} and in reflexive Banach spaces in \cite{At-Ch-Mou, Ri}, the weakest condition for the maximal monotonicity of the parallel sum available in the latter setting in the literature being recently introduced in \cite{Pen-Zal}. Since $S$ and $T$ are maximal monotone if and only if their inverse $S^{-1}$ and, respectively, $T^{-1}$ are maximal monotone, the sufficient conditions for the maximal monotonicity of $S||T$ in reflexive Banach spaces can be gathered from the ones formulated for the maximal monotonicity of the sum of two maximal monotone operators, applied to $S^{-1} + T^{-1}$.

When $Y$ is another real nonzero Banach space with $Y^*$ its topological dual space, $S:X\rightrightarrows X^*$ and $T:Y\rightrightarrows Y^*$ are two monotone operators and $A:X \rightarrow Y$ is a linear continuous mapping with adjoint mapping $A^*:Y^* \rightarrow X^*$, Penot and Z\u alinescu proposed in \cite{Pen-Zal} the following \textit{generalized parallel sums of $S$ and $T$ defined via $A$}
$$S||^A T:Y\rightrightarrows Y^*,\ S||^A T(y):=(A S^{-1}A^*+T^{-1})^{-1}(y) \ \forall y \in Y$$
and
$$S||_A T : X\rightrightarrows {X^*},\ S||_A T(x):= (S^{-1}+(A^*TA)^{-1})^{-1}(x) \ \forall x \in X,$$
respectively.
One can easily notice that when $X=Y$ and $A$ is the identity mapping on $X$, then they both collapse into $S||T$. As the monotonicity of $S$ and $T$ gives rise to the same property for $S||^A T$ and $S||_A T$, the question of how to guarantee for these maximal monotonicity, provided that $S$ and $T$ are maximal monotone, comes automatically.

This question was already addressed by Stephen Simons in \cite{Sim3} in general Banach spaces in what concerns the generalized parallel sum $S||^A T$. Under the assumption that $S$ and $T$ are maximal monotone operators of Gossez type (D), in the mentioned paper, interiority-type regularity conditions for ensuring that $S||^A T$ is a maximal monotone operator of Gossez type (D), too, have been formulated. Due to its nature,  at least in reflexive spaces, statements on the maximal monotonicity of the parallel sum $S||^A T$ and corresponding \textit{interiority-} and \textit{closedness-type} regularity conditions \textit{can be derived} from the statements given in the literature for the \textit{sum} of a monotone operator with the composition of a second one with a linear continuous mapping (see \cite{Pen-Zal1, Bo-Gra-Wa, Bo-Cse}).

On the other hand, the approach suggested above for $S||^A T$, regarding the direct derivation of sufficient conditions for maximal monotonicity from the already existent ones, cannot be applied to $S||_A T$ accordingly. This fact represented the starting point of the investigations made in this paper, where we want to provide \textit{interiority-} and \textit{closedness-type} regularity conditions for the maximal monotonicity of Gossez type (D) of $S||_A T$, whenever
$S$ and $T$ are maximal monotone operators of Gossez type (D).

The outline of the paper is the following. In the remaining of this section we recall some elements of convex analysis and introduce the necessary apparatus of notions and results referring to monotone operators in general Banach spaces. In Section 2 we investigate the fulfilment in an exact sense of a \textit{generalized bivariate infimal convolution formula} for which we provide, by making use of a special conjugate formula, equivalent closedness-type conditions, but also sufficient interiority-type ones. This formula represents the premise for ensuring in Section 3 maximal monotonicity of Gossez type (D) of a generalized parallel sum of the maximal monotone operators of Gossez type (D) $S$ and $T$, defined by making use of their extensions to the corresponding biduals. The maximal monotonicity of Gossez type (D) of $S||_A T$ will follow as a particular instance of this general result. A special attention will be also given to the formulation of further sufficient conditions for the interiority-type regularity condition and to the situation when these became equivalent. Finally, in Section 4, some particular instances, to which the general results on the maximal monotonicity of $S||_A T$  give rise, are considered.

\subsection{Elements of convex analysis}

Let $X$ a real separated locally convex space and $X^*$ its topological dual space. We denote by $w(X,X^*)$ (or, for short, $w$) the \textit{weak topology} on $X$ induced by $X^*$ and by $w(X^*,X)$ (or, for short, $w^*$) the weak$^*$ topology on $X^*$ induced by $X$. We denote by $\<x^*,x\>$ the value of the continuous linear functional $x^*\in X^*$ at $x\in X$. For a given set $D\subseteq X$, we denote by $\co D, \aff D, \inte D$ and $\cl D$, its \emph{convex hull, affine hull, interior} and \emph{closure}, respectively. When $Z \subseteq X$ is a given set we say that $D$ is \textit{closed regarding the set} $Z$ if $\cl D \cap Z = D \cap Z$. The \textit{conic hull} of the set $D$ will be denoted by $\cone D = \cup_{\lambda > 0} \lambda D$, while its \textit{relative interior} is defined as (see \cite{Zal-carte})
$$\ri D=\left\{
\begin{array}{ll}
\rint D, & \mbox {if } \aff D \mbox{ is a closed set},\\
\emptyset, & \mbox{otherwise},
\end{array}\right.$$
where $\rint D:=\inte_{\aff D} D$. The \emph{algebraic interior} (or \emph{core}) of $D$ is the set (see \cite{holmes, Rock-conj-dual, Zal-carte})
$$\core D=\{u\in X| \ \forall x\in X \ \exists \delta>0\mbox{ such that }\forall \lambda\in[0,\delta]:u+\lambda x\in D\},$$ while its \emph{relative algebraic interior} (or \emph{intrinsic core}) is the set (see \cite{holmes, Zal-carte})
$$\icr D=\{u\in X| \ \forall x\in\aff(D-D) \ \exists \delta>0\mbox{ such that }\forall \lambda\in[0,\delta]:u+\lambda x\in D\}.$$
One always has that $\rint D \subseteq \icr D$. The \emph{intrinsic relative algebraic interior} of $D$
(see \cite{Zal-carte, Zal-art-ic}) is defined as
$$^{\ic}{D}=\left\{
\begin{array}{ll}
\icr D, & \mbox {if } \aff D \mbox{ is a closed set},\\
\emptyset, & \mbox{otherwise}.
\end{array}\right.$$
Thus we have, in general, that
\begin{equation}\label{eqriic}
\ri D \subseteq {^{\ic}D}.
\end{equation}

In the case when $D$ is a convex set, the above generalized interiority notions can be characterized as follows:
\begin{enumerate}
\item[$\bullet$] $\core D=\{x\in D:\cone(D-x)=X\}$ (see \cite{Rock-conj-dual, Zal-carte});

\item[$\bullet$]$\icr D=\{x\in D:\cone(D-x)\mbox{ is a linear
subspace of }X\}$ (see \cite{holmes, Zal-carte});

\item[$\bullet$] ${^{\ic} D}=\{x\in D:\cone(D-x)\mbox{ is a closed
linear subspace of }X\}$ (see \cite{Zal-carte,  Zal-art-ic});

\item[$\bullet$] $x\in{^{\ic} D}$ if and only if $x\in\icr D$ and
$\aff(D-x)$ is a closed linear subspace of $X$ (see
\cite{Zal-carte, Zal-art-ic})
\end{enumerate}
and we have the following inclusions
\begin{equation}\label{int-not-1}
\inte D \subseteq \core D \subseteq {^{\ic} D}\subseteq \icr D \subseteq D,
\end{equation}
they being in general strict.

When $Y$ is another real separated locally convex space and $A:X\To Y$ a linear continuous mapping we consider the following notation $\Delta_X^A:=\{(x,Ax):x\in X\}$,
which becomes when $A=\id_X : X \To X$ with $\id_X(x)=x$ for all $x\in X$ (the \emph{identity mapping} on $X$) the \textit{diagonal subspace}
$\Delta_X :=\{(x,x): x \in X\} $ of $X \times X$. The following result, which is of interest independently of the purposes of this article,  will play an important role in the sequel.

\begin{lem}\label{l2.1} Let $X$ and $Y$ be separated locally convex spaces, $U\subseteq X$ and $V\subseteq Y$ two given convex sets
and $A:X \To Y$ a linear continuous mapping.
Then it holds
$$(0,0)\in {^{\ic}\left(U\times V-\Delta_X^A\right)}\Leftrightarrow 0\in {^{\ic}(V-A(U))}.$$
\end{lem}
\begin{proof} In the proof we use the following two characterizations:
$$(0,0)\in {^{\ic}\left(U\times V-\Delta_X^A\right)} \Leftrightarrow C:=\cone\left(U\times V-\Delta_X^A\right) \  \mbox{is a closed linear subspace of} \ X \times Y$$
and
$$0\in {^{\ic}(V-A(U))} \Leftrightarrow D:=\cone(V-A(U)) \ \mbox{is a closed linear subspace of} \ Y.$$

$"\Rightarrow"$ Suppose that $C$ is a closed linear subspace. Since $U$ and $V$ are convex sets, one has that $D$ is a convex cone. In order to proof that $D$ is a linear subspace, we show that $-D \subseteq D$. Take an arbitrary $d\in D$. Thus $d=\a(v-Au)$ for $\a>0,\, u\in U$ and $v\in V,$ hence $(0,d)=(\a(u-u),\a(v-Au))=\a((u,v)-(u,Au)) \in C.$ But $C$ is a linear space, hence $(0,-d)\in C,$ that is $(0,-d)=\b((u_1,v_1)-(x,Ax)),$ with $\b>0,\, u_1\in U,\, v_1\in V$ and $x\in X.$ It results that $u_1-x=0,$ hence $x=u_1 \in U$. Thus $-d=\b(v_1-Au_1)$ with $\b>0,\, u_1\in U,\,v_1\in V,$ hence $-d\in D.$

We prove next that $D$ is closed and consider therefore an arbitrary element $d \in \cl D.$ Thus there exist $(\l_\a)_{\a\in\mathfrak{I}}\subseteq \Real_+,$ $(u_\a)_{\a\in\mathfrak{I}}\subseteq U$ and $(v_\a)_{\a\in\mathfrak{I}}\subseteq V$ such that $d_\a=\l_\a(v_\a-Au_\a)\To d.$ But $(0,d_\a)\in C$ for all $\a\in\mathfrak{I}$ and $C$ is closed, thus $(0,d)=\b(u-x,v-Ax),$ with $\b>0,\, u\in U,\, v\in V$ and $x\in X.$ Hence, $x=u \in U$ and, consequently, $d=\b(v-Au)\in D.$

$"\Leftarrow"$ Suppose now that $D$ is a closed linear subspace. The convexity of the sets $U$ and $V$ guarantees that $C$ is a convex cone. Next we prove that $-C \subseteq C$ and consider to this end an arbitrary $c \in C$. Thus $c=\a(u-x,v-Ax),$ with $\a>0,\, u\in U,\, v\in V,\, x\in X.$ Hence, $c=\a(0,v-Au)+\a(u-x,A(u-x)).$ Obviously, $\a(v-Au)\in D$ and since $D$ is a linear space, we have $-\a(v-Au)=\b(v_1-Au_1),$ with $\b>0, u_1\in U$ and $v_1\in V.$ Thus $-c=\b(0,v_1-Au_1)-\a(u-x,A(u-x))=\b \Big(u_1-\big(u_1+\a/\b(u-x)\big), v_1-A\big(u_1+\a/\beta(u-x)\big)\Big) \in C.$

In order to show that $C$ is closed we consider an element $c:=(c_1,c_2) \in \cl C$ and show that $c\in C.$  Thus
there exist $(\l_\a)_{\a\in\mathfrak{I}}\subseteq \Real_+,$ $(u_\a)_{\a\in\mathfrak{I}}\subseteq U,$ $(v_\a)_{\a\in\mathfrak{I}}\subseteq V$ and $(x_\a)_{\a\in\mathfrak{I}}\subseteq X$ such that $c_\a=\l_\a(u_\a-x_\a,v_\a-Ax_\a)\To c=(c_1,c_2).$ Obviously, $\l_\a(u_\a-x_\a)\To c_1,$ hence $\l_\a A(u_\a-x_\a)\To Ac_1$ and from here we obtain that $\l_{\a}(v_\a-Au_\a)\To c_2-Ac_1.$ But $\l_{\a}(v_\a-Au_\a)\in D$ for all $\a\in\mathfrak{I}$ and $D$ is closed, hence $c_2-Ac_1=\b(v-Au),$ with $\b>0,\, u\in U$ and $v\in V.$ Thus $(c_1,c_2)= \Big(u-\big(u-1/\b c_1\big),v-A\big(u-1/\b c_1\big)\Big) \in C$ and this concludes the proof.
\end{proof}

The \emph{indicator function} of a set $D \subseteq X$ is defined as $\delta_D:X\To\oReal:=\Real \cup \{\pm \infty\}$,
$$\delta_D(x)=\left\{\begin{array}{ll}
0, & \mbox {if } x\in D,\\
+\infty, & \mbox{otherwise}.
\end{array}\right.$$

For $E$ and $F$ two nonempty sets we consider
the \emph{projection operator} $\pr_E:E\times F\rightarrow E$,
$\pr_E(e,f)=e$ for all $(e,f)\in E\times F$. For $G$ and $H$ two further nonempty sets and
$k:E \rightarrow G$ and $l:F \rightarrow H$ two given functions we denote by $k \times l : E \times F \rightarrow G \times H$
the function defined as $k \times l(e,f) = (k(e),l(f))$ for all $(e,f) \in E \times F$. Throughout the paper, when an infimum is attained we write min instead of
inf.

Having a  function $f:X\To \oReal$ we denote its \textit{domain} by $\dom f=\{x\in X:f(x)<+\infty\}$ and its \textit{epigraph} by $\epi f=\{(x,r)\in X\times\Real:f(x)\leq r\}$. We call $f$ \emph{proper} if $\dom f\neq\emptyset$ and $f(x)>-\infty$ for all $x\in X$. By $\cl f : X \rightarrow \oReal$ we denote the \emph{lower semicontinuous hull} of $f$, namely the function whose epigraph is the closure of $\epi f$, that is $\epi(\cl f)=\cl(\epi f).$ We consider also $\co f : X \rightarrow \oReal$, the \emph{convex hull} of $f$, which is the greatest convex function majorized by $f$. For $x\in X$ such that
$f(x)\in\Real$ we define the \emph{subdifferential} of
$f$ at $x$ by $$\partial f(x)=\{x^*\in X^*:f(y)-f(x)\geq\<x^*,y-x\>\ \forall y\in X\}.$$
When $f(x) \in \{\pm\infty\}$ we take by convention $\partial f(x)=\emptyset$.

The \emph{Fenchel-Moreau conjugate} of $f$ is the function $f^*:X^*\To\oReal$ defined by
$$f^*(x^*)=\sup\limits_{x\in X}\{\<x^*,x\>-f(x)\} \ \forall x^*\in X^*.$$
One always has the \emph{Young-Fenchel inequality}
$$f^*(x^*)+f(x)\geq\< x^*,x\> \ \forall x\in X \ \forall x^*\in X^*.$$

Consider $Y$ another separated locally convex space and a mapping $h:X \To Y$. We denote by $h(D) =\{h(x):x\in D\}$
the \emph{image} of a set $D \subseteq X$ through $h$ and by $h^{-1}(E)=\{x\in X:h(x)\in E\}$ the \textit{inverse} of a set $E \subseteq Y$ through $h$.

For $A:X\To Y$ a linear continuous mapping, $\im A:= A(X)$ denotes the \textit{image space} of $A$, while its \emph{adjoint operator} $A^*:Y^*\To X^*$ is defined by $\<A^*y^*,x\>=\< y^*,Ax\>$ for all $y^*\in Y^*$ and $x\in X$. When $X$ and $Y$ are normed spaces, the \textit{biadjoint operator} of $A$, $A^{**}:X^{**}\To Y^{**}$, is  defined as being the adjoint operator of $A^*$.

\subsection{Monotone operators in general Banach spaces}

Consider further $X$ a nonzero real Banach space, $X^*$ its topological dual space and $X^{**}$ its topological bidual space. Throughout the paper we identify $X$ with its image under the canonical injection of $X$ into $X^{**}$. A multivalued operator $S:X\rightrightarrows X^*$ is said to be \emph{monotone} if
$$\< y^*-x^*,y-x\>\geq 0,\mbox{ whenever } y^*\in S(y)\mbox{ and }x^*\in S(x).$$
A monotone operator $S$ is called \emph{maximal monotone} if its graph $G(S)$ is not
properly contained in the graph of any other monotone operator $S':X\rightrightarrows X^*$. For the operator $S$ we consider also its \emph{domain} $D(S):=\{x\in X:S(x)\neq\emptyset\}=\pr_X(G(S))$ and its \emph{range} $R(S):=\cup_{x\in X}S(x)=\pr_{X^*}(G(S))$. The most prominent example of a maximal monotone operator is the subdifferential of a proper, convex and lower semicontinuous function (see \cite{Roc1}). However, there exist maximal monotone operators which are not subdifferentials (see \cite{sim1, Sim}).

To an arbitrary monotone operator $S:X\rightrightarrows X^*$ we associate the \emph{Fitzpatrick function} $\varphi_S:X\times X^*\To \oReal$, defined by
$$\varphi_S(x,x^*)=\sup\{\< y^*,x\> +\<x^*,y\>-\< y^*,y\>:y^*\in S(y)\},$$ which is obviously convex and weak$\times$weak$^*$ lower semicontinuous. Introduced by Fitzpatrick in 1988 (see \cite{fitz}) and rediscovered after some years in \cite{bu-sv-02, legaz-thera2}, it proved to be very important in the theory of maximal monotone operators, revealing important connections between convex analysis and monotone operators (see \cite{Bausch, Borwein, Bot, Bo-Cse, Bo-Gra-Wa, BGW-max-suma, bu-fitz, bu-sv-03, bu-sv-02, legaz-sv, penot, Pen-Zal, Sim, simons-zal, zalproc, voisei, penot-is-conv, penot-repr} and the references therein).

Denoting by $c:X\times X^*\rightarrow\Real$, $c(x,x^*)=\<x^*,x\>$ for all $(x,x^*)\in X\times X^*$ the \textit{coupling function} of $X \times X^*$, one can easily show that $\varphi_S(x,x^*)=c_S^*(x^*,x)$ for all $(x,x^*)\in X\times X^*$, where $c_S : X\times X^*\rightarrow\oReal$, $c_S=c+\delta_{G(S)}$. Well-linked to the Fitzpatrick function is the function $\psi_S : X\times X^*\rightarrow\oReal$, $\psi_S=\cl_{\|\cdot\|\times\|\cdot\|_*}(\co c_S)$, where the closure is taken in the strong topology of $X\times X^*$. For all $(x,x^*) \in X \times X^*$ we have $\psi_S^*(x^*,x) = \varphi_S(x,x^*)$, while when $X$ is a reflexive Banach space the equality $\varphi_S^*(x^*,x)=\psi_S(x,x^*)$ holds (see \cite[Remark 5.4]{bu-sv-02}). The most important properties of the Fitzpatrick function of a maximal monotone operator follow.
\begin{lem} \label{fitz}(see \cite{fitz}) Let $S:X\rightrightarrows X^*$ be a maximal monotone operator. Then
\begin{enumerate}
\item[(i)]$\varphi_S(x,x^*)\geq\< x^*,x\>$ for all $(x,x^*)\in X\times X^*$,
\item[(ii)] $G(S)=\{(x,x^*)\in X\times X^*:\varphi_S(x,x^*)=\< x^*,x\>\}$.
\end{enumerate}
\end{lem}

They gave rise to the following notion introduced in connection to a monotone operator.

\begin{defn}\label{repr} For $S:X\rightrightarrows X^*$ a monotone operator, we call \emph{representative function} of $S$ a convex
and lower semicontinuous (in the strong topology of $X\times X^*$) function $h_S:X\times X^*\To\oReal$ fulfilling
\begin{equation*}
h_S\geq c\mbox{ and }G(S)\subseteq\{(x,x^*)\in
X\times X^*:h_S(x,x^*)=\< x^*,x\>\}.
\end{equation*}
\end{defn}
If $G(S)\neq\emptyset$ (which is the case when $S$ is maximal monotone), then every representative function of $S$ is proper. Obviously, the Fitzpatrick function associated to a maximal monotone operator is a representative function of the operator. From \cite{bu-sv-02} we have the following properties for the representative function of a maximal monotone operator.

\begin{prop} \label{frepr} Let $S:X\rightrightarrows X^*$ be a maximal monotone operator and $h_S$ be a representative function of $S$.
Then the following statements are true:
\begin{enumerate}
\item[(i)]$\varphi_S\leq h_S\leq \psi_S$;
\item[(ii)]the function $(x,x^*) \mapsto h_S^*(x^*,x)$ is also a representative function of $S$;
\item[(iii)] $\{(x,x^*)\in X\times X^*:h_S(x,x^*)=\<x^*,x\>\}=\{(x,x^*)\in X\times X^*:h_S^{*}(x^*,x)=\<x^*,x\>\}=G(S)$.
\end{enumerate}
\end{prop}

By Proposition \ref{frepr} it follows that a convex and lower semicontinuous function $f:X\times X^*\To\oReal$ is a representative function
of the maximal monotone operator $S$ if and only if $\varphi_S\leq f\leq \psi_S$, in particular, $\varphi_S$ and $\psi_S$ are representative functions of $S$.
Let us also notice that if $f:X\To \oReal$ is a proper, convex and lower semicontinuous function, then a representative function of
the maximal monotone operator $\partial f:X\rightrightarrows X^*$ is the function $(x,x^*)\mapsto f(x)+f^*(x^*)$. Moreover, according to \cite[Theorem 3.1]{bu-fitz} (see also \cite[Example 3]{penot-repr}), if $f$ is a sublinear and lower semicontinuous function, then the operator $\partial f:X\rightrightarrows X^*$ has a unique representative function, namely the function $(x,x^*)\mapsto f(x)+f^*(x^*)$. For more on the properties of representative functions we refer to \cite{Borwein, bu-sv-02, legaz-sv, Pen-Zal} and the references therein.

Next we give a maximality criteria for a monotone operator valid in reflexive Banach spaces (cf. \cite[Theorem 3.1]{bu-sv-03} and \cite[Proposition 2.1]{Pen-Zal}; see also \cite{sim1} for other maximality criteria in reflexive spaces).

\begin{thm}\label{max-cr}  Let $X$ be a reflexive Banach space and $f:X\times X^*\To \oReal$ a proper, convex and lower semicontinuous function such that $f\geq c$. Then the operator whose graph is the set $\{(x,x^*)\in X\times X^*:f(x,x^*)=\langle x^*,x\rangle \}$ is maximal monotone if and only if $f^*(x^*,x) \geq \langle x^*,x \rangle $ for all $(x,x^*) \in X \times X^*$.
\end{thm}

For the following  generalization of this result to general Banach spaces we refer to \cite[Theorem 4.2]{mas}.

\begin{thm}\label{max-cr-nref} Let $X$ be a nonzero Banach space and $f:X\times X^*\To \oReal$ a proper, convex and lower semicontinuous function such that $f\geq c$ and $f^*(x^*,x^{**}) \geq \< x^{**}, x^* \>$ for all $(x^*,x^{**}) \in X^* \times X^{**}$. Then the operator whose graph is the set $\{(x,x^*)\in X\times X^*:f(x,x^*)=\langle x^*,x\rangle \}$ is maximal monotone and it holds $\{(x,x^*)\in X\times X^*:f(x,x^*)=\langle x^*,x\rangle \} = \{(x,x^*)\in X\times X^*:f^*(x^*,x)=\langle x^*,x\rangle \}$.
\end{thm}

In the last part of this section we turn our attention to a particular class of maximal monotone operators on general Banach spaces.

\begin{defn}\label{GossezD}{(see \cite{Gos})} Let $S:X\rightrightarrows X^*$ be a maximal monotone operator.
\begin{enumerate}
\item[(a)] Gossez's monotone closure of $S$ is the operator $\overline{S}:X^{**}\rightrightarrows X^*$ whose graph is
$$G(\overline{S})=\{(x^{**},x^*)\in X^{**}\times X^*:\<x^{**}-y,x^{*}-y^*\>\ge 0 \ \forall(y,y^*)\in G(S)\}.$$
\item[(b)] The operator $S:X\rightrightarrows X^*$ is said to be of Gossez type (D) if for any $(x^{**},x^*)\in G(\overline{S})$ there exists a bounded net $\{(x_\a,x_\a^*)\}_{\a\in \mathfrak{I}}\subseteq G(S)$ which converges to $(x^{**},x^*)$ in the $w^*\times\|\cdot\|_*$-topology of $X^{**}\times X^*.$
\end{enumerate}
\end{defn}
Gossez proved in \cite{Gos2} that a maximal monotone operator $S:X\rightrightarrows X^*$ of Gossez type (D) has a \textit{unique}
maximal monotone extension to the bidual, namely, its \textit{Gossez's monotone closure} $\overline{S}:X^{**}\rightrightarrows X^*$. The following characterization of the maximal monotone operators of Gossez type (D) was recently provided in \cite{mas3} (see also \cite{mas2}).

\begin{thm}\label{th-gos}  Let $X$ be a nonzero real Banach space and $S:X \rightrightarrows X^*$ a maximal monotone operator. The following statements are equivalent:
\begin{enumerate}
\item[(a)] $S$ is of Gossez type (D);
\item[(b)] $S$ is of Simons negative infimum type (NI) (see \cite{Sim2}), namely
$$\inf_{(y,y^*)\in G(S)}\<y-x^{**},y^*-x^{*}\>\le 0 \ \forall(x^*,x^{**})\in X^*\times X^{**};$$
\item[(c)] there exists a representative function $h_S$ of $S$ such that
$$h_S^*(x^*,x^{**}) \geq \< x^{**}, x^* \> \ \forall (x^*,x^{**}) \in X^* \times X^{**};$$
\item[(d)] for every representative function $h_S$ of $S$ one has
$$h_S^*(x^*,x^{**}) \geq \< x^{**}, x^* \> \ \forall (x^*,x^{**}) \in X^* \times X^{**}.$$
\end{enumerate}
\end{thm}
A representative function $h_S$ of a maximal monotone operator $S:X \rightrightarrows X^*$ fulfilling the inequality in the item (c) (or (d)) of the above theorem is called \textit{strong representative} function of $S$ (see \cite{voiseizal}). The Fitzpatrick function $\varphi_S$ of a maximal monotone operator $S:X \rightrightarrows X^*$ of Gossez type (D) is a strong representative function and one has $\varphi_{\overline S}|_{X \times X^*} = \varphi_S$. When $h_S:X \times X^* \rightarrow \oReal$ is a representative function of a maximal monotone operator of Gossez type (D) $S:X \rightrightarrows X^*$, then $h_S^* : X^* \times X^{**} \rightarrow \oReal$ is a representative function of the inverse operator $\overline S^{-1} : X^* \rightarrow X^{**}$ of Gossez's monotone closure $\overline{S}$ of $S$ (for these statements we refer the reader to \cite{mas3}).

\section{A generalized bivariate infimal convolution formula}

In this section we provide, by making use of an appropriate \textit{conjugate formula}, sufficient conditions for an \textit{extended bivariate infimal convolution formula}, which we use in the sequel.

\subsection{An useful conjugate formula}

Let $X,Y,Z$ be real separated locally convex spaces with topological duals $X^*, Y^*$ and $Z^*,$ respectively.

\begin{thm}\label{thcomp}
Let $f:X\To\overline{\Real}$ and $g:Y\To\overline{\Real}$ be proper, convex and lower semicontinuous functions and $A:Z\To X$ and $B:Z\To Y$ be linear continuous mappings such that $A^{-1}(\dom f)\cap B^{-1}(\dom g)\neq \emptyset$.
\begin{enumerate}
\item[(a)] For every set $U \subseteq Z^*$ the following statements are equivalent:
\begin{enumerate}
\item[(i)] The set $\{(A^*x^*+B^*y^*,r): r \in \Real, f^*(x^*)+g^*(y^*)\le r\}$ is closed regarding $U\times \Real$ in $(Z^*,w^*)\times \Real$;
\item[(ii)] $(f \circ A + g \circ B)^*(z^*) = \min \{f^*(x^*)+g^*(y^*):(x^*,y^*)\in X^*\times Y^*,A^*x^*+B^*y^*=z^*\}$ for all $z^*\in U.$
\end{enumerate}
\item[(b)] If $X,Y$ and $Z$ are Fr\'echet spaces and
$$(0,0) \in {^{\ic}(\dom f \times \dom g - (A\times B)(\Delta_Z))},$$
then the statements (i) and (ii) are valid for every $U \subseteq Z^*$.
\end{enumerate}
\end{thm}

\begin{proof}
(a) Consider an arbitrary set $U \subseteq Z^*$ and the perturbation function
$$\Phi:Z\times X\times Y\To\oReal,\, \Phi(z,x,y)=f(Az+x)+g(Bz+y),$$
which is proper, convex and lower semicontinuous and fulfills
$$(0,0) \in \pr\nolimits_{X \times Y}(\dom \Phi) = \dom f \times \dom g - (A\times B)(\Delta_Z).$$
Its conjugate function looks for all $(z^*,x^*,y^*) \in Z^* \times X^* \times Y^*$ like
$$\Phi^*(z^*,x^*,y^*)=\delta_{\{0\}}(z^*-A^*x^*-B^*y^*)+f^*(x^*)+g^*(y^*).$$
Thus (ii) is nothing else than
$$(\Phi(\cdot,0,0))^*(z^*)=\min_{(x^*,y^*)\in X^* \times Y^*}\Phi^*(z^*,x^*,y^*) \ \forall z^* \in U.$$
According to \cite[Theorem 2]{Bo-Cse}, this is further equivalent to
\begin{equation}
\pr\nolimits_{Z^* \times \Real}(\epi \Phi^*) \ \mbox{is closed regarding} \ U\times \Real \ \mbox{in} \ (Z^*,w^*)\times \Real.
\end{equation}
As one can easily see, it holds
$$\pr\nolimits_{Z^* \times \Real}(\epi \Phi^*) = \{(A^*x^*+B^*y^*,r):f^*(x^*)+g^*(y^*)\le r\}$$
and in this way the equivalence (i)$\Leftrightarrow$(ii) is proven.

(b) Since $X,Y$ and $Z$ are Fr\'echet spaces and $(0,0) \in {^{\ic}\left(\pr\nolimits_{X \times Y}(\dom \Phi) \right)}$,
by \cite[Corollary 2.7.3]{Zal-carte} it follows that for all $z^* \in Z^*$
$$(\Phi(\cdot,0,0))^*(z^*)=\min_{(x^*,y^*)\in X^* \times Y^*}\Phi^*(z^*,x^*,y^*)$$
or, equivalently,
$$(f \circ A + g \circ B)^*(z^*) = \min \{f^*(x^*)+g^*(y^*):(x^*,y^*)\in X^*\times Y^*,A^*x^*+B^*y^*=z^*\},$$
which concludes the proof.
\end{proof}

\begin{rmk}\label{r2.1}\rm
In the hypotheses of Theorem \ref{thcomp}, when $X$, $Y$ and $Z$ are Fr\'echet spaces, then, according to \cite[Proposition 2.7.2]{Zal-carte},
$${^{\ic}(\dom f \times \dom g - (A\times B)(\Delta_Z))} = \ri(\dom f \times \dom g - (A\times B)(\Delta_Z)).$$
\end{rmk}

\begin{rmk}\label{r2.2}\rm
We refer the reader to \cite{Bot} for examples where, even $X$, $Y$ and $Z$ are finite dimensional spaces,
the statements (i) and (ii) in Theorem \ref{thcomp}(a) are fulfilled, while the interiority-type condition in Theorem \ref{thcomp}(b)
fails.
\end{rmk}

\begin{rmk}\label{r2.3}\rm According to the previous theorem, one obtains when $Z=X$, $A=\id_X$ and $X$ and $Y$ are Fr\'echet spaces as a sufficient condition for the
exact conjugate formula
\begin{equation}\label{eqgA}
(f + g \circ B)^*(z^*) = \min \{f^*(z^* - B^*y^*)+g^*(y^*):y^*\in Y^*\} \ \forall z^*\in X^*
\end{equation}
the interiority-type condition
$$(0,0) \in {^{\ic}(\dom f \times \dom g - \Delta_Z^B)}.$$
Via Lemma \ref{l2.1} it follows that this is nothing else than
$$0 \in {^{\ic} \left(\dom g - B(\dom f)\right)},$$
which is a regularity condition for \eqref{eqgA} that has been already considered in literature (see, for instance, \cite{Zal-carte}).
\end{rmk}

\subsection{A bivariate infimal convolution formula adequate for the generalized parallel sum}

Let $X$ and $Y$ be two Banach spaces with  $X^*$ and $Y^*$ their topological dual spaces and $X^{**}$ and $Y^{**}$ their topological bidual spaces, respectively.
Further, let $f:X\times X^*\To\overline{\Real}$ and $g:Y\times Y^*\To\overline{\Real}$ be two given functions and $A:X\To Y$ a linear continuous mapping. In this subsection we deal with the following \textit{extended bivariate infimal convolutions} $f\bigcirc_1^A g:X\times X^*\To\oReal$,
$$(f\bigcirc_1^A g)(x,x^*)=\inf\{f(u,x^*)+g(Aw,v^*):u,w\in X, v^*\in Y^*, u+w=x,\, A^*v^*=x^*\},$$
and $f^*\bigcirc_2^A g^*:X^*\times X^{**}\To\oReal$,
$$(f^*\bigcirc_2^A g^*)(x^*,x^{**})=\inf\{f^*(x^*,u^{**})+g^*(v^*,A^{**}w^{**}): u^{**},w^{**}\in X^{**}, v^*\in Y^*, u^{**}+w^{**}=x^{**},\, A^*v^*=x^*\},$$
respectively. By making use of Theorem \ref{thcomp}, we can prove the following result.

\begin{thm}\label{t2.5} Assume that $f:X\times X^*\To\overline{\Real}$ and $g:Y\times Y^*\To\overline{\Real}$ are  proper, convex and lower semicontinuous functions such that $\dom g \times \pr\nolimits_{X^*}(\dom f) \cap \im A \times\Delta_{Y^*}^{A^*}\neq\emptyset$.
\begin{enumerate}
\item[(a)] The following statements are equivalent:
\begin{enumerate}
\item[(i)] The set $\{(u^*,A^*v^*,A^{**}u^{**}+v^{**},r): r \in \Real, f^*(u^*,u^{**})+g^*(v^*,v^{**})\le r\}$ is closed regarding $\Delta_{X^*}\times \im A^{**} \times \Real$ in  $(X^*,w^*)\times (X^*,w^*)\times (Y^{**},w^*)\times \Real$;
\item[(ii)] $(f\bigcirc_1^A g)^*(x^*,x^{**})=(f^*\bigcirc_2^A g^*)(x^*,x^{**})$ and $f^*\bigcirc_2^A g^*$ is exact (that is, the infimum in the definition of $(f^*\bigcirc_2^A g^*)(x^*,x^{**})$ is attained) for every $(x^*,x^{**})\in X^*\times X^{**}.$
\end{enumerate}
\item[(b)] If
$$(0,0,0) \in {^{\ic}\left(\dom g \times \pr\nolimits_{X^*}(\dom f)-\im A \times\Delta_{Y^*}^{A^*}\right)},$$
then the statements (i) and (ii) are true.
\end{enumerate}
\end{thm}

\begin{proof}
Consider the proper, convex and lower semicontinuous functions $F:X\times X\times X^*\To\oReal,$ $\,F(u,w,u^*)=f(u,u^*)$ and $G:X\times Y\times Y^*\To\oReal,\, G(u,v,v^*)=g(v,v^*)$ and the linear continuous mappings $M:X \times X \times Y^* \rightarrow X\times X\times X^*,$ $M= \id_X\times \id_X\times A^*$, and $N : X \times X \times Y^* \rightarrow X\times Y\times Y^*$, $N=\id_X\times A\times \id_{Y^*}$. Since $\dom g \times \pr\nolimits_{X^*}(\dom f) \cap \im A \times\Delta_{Y^*}^{A^*}\neq\emptyset$, we obtain that $M^{-1}(\dom F)\cap N^{-1}(\dom G)\neq\emptyset.$

(a) According to Theorem \ref{thcomp}(a), applied for $U:=\Delta_{X^*}\times \im A^{**} \subseteq X^*\times X^*\times Y^{**}$, we have that
\begin{equation}\label{interm1}
\begin{array}{l}
\{(M^*(u_1^*,w^*,u^{**})+N^*(u_2^*,v^*,v^{**}),r): r \in \Real, F^*(u_1^*,w^*,u^{**})+G^*(u_2^*,v^*,v^{**})\le r\} \ \mbox{is closed regarding}\\
\Delta_{X^*}\times \im A^{**} \times \Real \ \mbox{in} \ (X^*,w^*)\times (X^*,w^*)\times (Y^{**},w^*)\times \Real
\end{array}
\end{equation}
if and only if
\begin{equation}\label{interm2}
\begin{array}{c}
(F \circ M + G \circ N)^*(x^*,x^*,A^{**}x^{**}) =\\
\min\limits_{\stackrel{(u_1^*,w^*,u^{**})\in X^*\times X^*\times X^{**}}{(u_2^*,v^*,v^{**})\in X^*\times Y^*\times Y^{**}}}\{F^*(u_1^*,w^*,u^{**})+G^*(u_2^*,v^*,v^{**}):M^*(u_1^*,w^*,u^{**})+N^*(u_2^*,v^*,v^{**})=(x^*,x^*,A^{**}x^{**})\}\\
\mbox{for all} \ (x^*,x^{**})\in  X^* \times X^{**}.
\end{array}
\end{equation}
Since  $F^*(u^*,w^*,u^{**})=\d_{\{0\}}(w^*)+f^*(u^*,u^{**})$ for all $(u^*,w^*,u^{**}) \in X^* \times X^* \times X^{**}$ and
$G^*(u^*,v^*,v^{**})=\d_{\{0\}}(u^*)+g^*(v^*,v^{**})$ for all $(u^*,v^*,v^{**}) \in X^* \times Y^* \times Y^{**}$, one can easily see that
$$\{(M^*(u_1^*,w^*,u^{**})+N^*(u_2^*,v^*,v^{**}),r): r \in \Real, F^*(u_1^*,w^*,u^{**})+G^*(u_2^*,v^*,v^{**})\le r\} =$$ $$\{(u^*,A^*v^*,A^{**}u^{**}+v^{**},r):f^*(u^*,u^{**})+g^*(v^*,v^{**})\le r\},$$
which means that the statement in \eqref{interm1} is nothing else than (i).

On the other hand, for all $(x^*,x^{**})\in  X^* \times X^{**}$ it holds
$$(F \circ M + G \circ N)^*(x^*,x^*,A^{**}x^{**})=$$
$$\sup_{(u,w,v^*)\in X\times X\times Y^*}\{\<(x^*,x^*,A^{**}x^{**}),(u,w,v^*)\>-(F\circ M)(u,w,v^*)-(G\circ N)(u,w,v^*)\}=$$
$$\sup_{(u,w,v^*)\in X\times X\times Y^*}\{\<(x^*,x^{**}),(u+w,A^*v^*)\>-f(u,A^*v^*)-g(Aw,v^*)\}=$$
$$\sup_{(s,s^*)\in X\times X^*} \left \{\<(x^*,x^{**}),(s,s^*)\>-\inf_{(u,w,v^*)\in X\times X\times Y^*}\{f(u,s^*)+g(Aw,v^*):\, u+w=s,\, A^*v^*=s^*\}\right\}=$$
$$(f\bigcirc_1^A g)^*(x^*,x^{**})$$
and
$$\min\limits_{\stackrel{(u_1^*,w^*,u^{**})\in X^*\times X^*\times X^{**}}{(u_2^*,v^*,v^{**})\in X^*\times Y^*\times Y^{**}}}\{F^*(u_1^*,w^*,u^{**})+G^*(u_2^*,v^*,v^{**}):M^*(u_1^*,w^*,u^{**})+N^*(u_2^*,v^*,v^{**})=(x^*,x^*,A^{**}x^{**})\}=$$
$$\min_{\stackrel{(u^*,0,u^{**})\in X^*\times X^*\times X^{**}}{(0,v^*,v^{**})\in X^*\times Y^*\times Y^{**}}}\{f^*(u^*,u^{**})+g^*(v^*,v^{**}):M^*(u^*,0,u^{**})+N^*(0,v^*,v^{**})=(x^*,x^*,A^{**}x^{**})\}=$$
$$\min_{(u^{**},w^{**},v^*)\in X^{**}\times X^{**}\times Y^*}\{f^*(x^*,u^{**})+g^*(v^*,A^{**}w^{**}):A^*v^*=x^*,\, u^{**}+w^{**}=x^{**}\}=$$
$$(f^*\bigcirc_2^A g^*)(x^*,x^{**}),$$
which means that the the statement in \eqref{interm2} says actually that $(f\bigcirc_1^A g)^*(x^*,x^{**})=(f^*\bigcirc_2^A g^*)(x^*,x^{**})$ and $f^*\bigcirc_2^A g^*$ is exact for every $(x^*,x^{**})\in X^*\times X^{**}$. This leads to the desired conclusion.

(b) The assertion is a direct consequence of Theorem \ref{thcomp}(b), as, obviously,
$$(0,0,0,0,0,0) \in {^{\ic}\left(\dom F\times \dom G- (M\times N)(\Delta_{X\times X\times Y^*})\right)} \Leftrightarrow$$
$$(0,0,0,0,0,0) \in {^{\ic}\left(X \times X \times X \times  \left (\dom g \times \pr\nolimits_{X^*}(\dom f)-\im A \times\Delta_{Y^*}^{A^*} \right)\right)} \Leftrightarrow$$
$$(0,0,0) \in {^{\ic}\left(\dom g \times \pr\nolimits_{X^*}(\dom f)-\im A \times\Delta_{Y^*}^{A^*}\right)}.$$
\end{proof}

\begin{rmk}\label{r2.4}\rm
In the hypotheses of Theorem \ref{t2.5} and by keeping the notations used in its proof, according to Remark \ref{r2.1}, we have
$${^{\ic}\left(\dom F \times \dom G - (M \times N)(\Delta_{X\times X\times Y^*})\right)} = \ri\left(\dom F \times \dom G - (M\times N)(\Delta_{X\times X\times Y^*})\right),$$
which is equivalent to
$${^{\ic}\left(\dom g \times \pr\nolimits_{X^*}(\dom f)-\im A \times\Delta_{Y^*}^{A^*}\right)} = \ri\left(\dom g \times \pr\nolimits_{X^*}(\dom f)-\im A \times\Delta_{Y^*}^{A^*}\right).$$
\end{rmk}

In reflexive Banach spaces the equivalence in Theorem \ref{t2.5}(a) gives rise to the following result.

\begin{cor}\label{c2.1}Let $X$ and $Y$ be reflexive Banach spaces and $f:X\times X^*\To\overline{\Real}$ and $g:Y\times Y^*\To\overline{\Real}$ proper, convex and lower semicontinuous functions such that $\pr_{X^*}(\dom f)\cap A^*(\pr_{Y^*}(\dom g))\neq\emptyset$. Then the following statements are equivalent:
\begin{enumerate}
\item[(i)] the set $\{(u^*,A^*v^*,Au+v,r): r \in \Real, f^*(u^*,u)+g^*(v^*,v)\le r\}$ is closed regarding $\Delta_{X^*}\times \im A \times \Real$ in  $(X^*,\|\cdot\|_{*})\times (X^*,\|\cdot\|_{*})\times (Y,\|\cdot\|)\times \Real$;
\item[(ii)] $(f\bigcirc_1^A g)^*(x^*,x)=(f^*\bigcirc_2^A g^*)(x^*,x)$ and $f^*\bigcirc_2^A g^*$ is exact for every $(x^*,x)\in X^*\times X.$
\end{enumerate}
\end{cor}

\section{The maximal monotonicity of Gossez type (D) of $S||_A T$}

In what follows we assume that $X$ and $Y$ are real nonzero Banach spaces, that $S:X\rightrightarrows X^*$ and $T:Y\rightrightarrows Y^*$ are two monotone operators and that $A:X\To Y$ is a linear continuous mapping. For $\overline{S} : X^{**} \rightrightarrows X^*$ and $\overline T:Y^{**} \rightrightarrows Y^*$, Gossez's monotone closures of $S$ and $T$, respectively, we consider their \textit{extended generalized parallel sum defined via $A$}, which is the multivalued operator defined as
$$\overline{S}||_A\overline{T}:X\rightrightarrows X^*,\overline{S}||_A\overline{T}(x):=(\overline{S}^{-1}+(A^*\overline{T}A^{**})^{-1})^{-1}(x) \ \forall x \in X.$$
The following result proposes two sufficient conditions ensuring the maximal monotonicity of Gossez type (D) of $\overline{S}||_A\overline{T}$, provided that both operators are maximal monotone of Gossez type (D), and it will give rise to a characterization of the maximal monotonicity of the \textit{generalized parallel sum of $S$ and $T$ defined via $A$},
$$S||_A T:X\rightrightarrows {X^*},\,S||_A T(x):= (S^{-1}+(A^*TA)^{-1})^{-1}(x) \ \forall x \in X.$$

\begin{thm}\label{t3.1} Let $S:X\rightrightarrows X^*$ and $T:Y\rightrightarrows Y^*$ be two maximal monotone operators of Gossez type (D) with strong representative functions $h_S$ and $h_T$, respectively, and $A:X\To Y$ a linear continuous mapping such that $\dom h_T \times \pr_{X^*}(\dom h_S) \cap \im A \times\Delta_{Y^*}^{A^*}\neq\emptyset$. Assume that one of the following conditions is fulfilled:
\begin{enumerate}
\item[(a)] $(0,0,0)\in {^{\ic}(\dom h_T \times \pr_{X^*}(\dom h_S)-\im A \times\Delta_{Y^*}^{A^*})}$;
\item[(b)] the set $\{(u^*,A^*v^*,A^{**}u^{**}+v^{**},r): r \in \Real, h_S^*(u^*,u^{**})+h_T^*(v^*,v^{**})\le r\}$ is closed regarding $\Delta_{X^*}\times \im A^{**}\times \Real$ in  $(X^*,w^*)\times (X^*,w^*)\times (Y^{**},w^*)\times \Real$.
\end{enumerate}
Then the function $h:X\times X^*\To\oReal,\, h(x,x^*)= \cl\nolimits_{\|\cdot\|\times\|\cdot\|_*}(h_S\bigcirc_1^A h_T)(x,x^*)$, is a strong representative function of $\overline{S}||_A\overline{ T}$ and the extended generalized parallel sum $\overline{S}||_A \overline{T}$ is a maximal monotone operator of Gossez type (D).
\end{thm}

\begin{proof} Obviously, $h:X\times X^*\To\oReal$ is convex and (strong) lower semicontinuous and, due to the feasibility condition $\dom h_T \times \pr_{X^*}(\dom h_S) \cap \im A \times\Delta_{Y^*}^{A^*}\neq\emptyset$, $h$ is not identical to $+\infty$. Since one of the conditions (a) and (b) is fulfilled, then one has, via Theorem \ref{t2.5}, that $h^*(x^*,x^{**})=(h_S\bigcirc_1^A h_T)^*(x^*,x^{**})=(h_S^*\bigcirc_2^A h_T^*)(x^*,x^{**})$ and $h_S^*\bigcirc_2^A h_T^*$ is exact for every $(x^*,x^{**}) \in X^*\times X^{**}$.

Take an arbitrary $(x,x^*) \in X \times X^*$. Then we have
$$(h_S\bigcirc_1^A h_T)(x,x^*)=\inf \{h_S(u,x^*)+h_T(Aw,v^*):u,w\in X, v^* \in Y^*, u+w=x,\, A^*v^*=x^*\}$$
$$\geq \inf \{\<x^*,u\>+\<x^*,w\>:u,w\in X, u+w=x\}=\<x^*,x\>.$$
Hence, $h(x,x^*) = \cl_{\|\cdot\|\times\|\cdot\|_*}(h_S\bigcirc_1^A h_T)(x,x^*)\ge \<x^*,x\>$, which implies that $h \geq c$, concomitantly ensuring that $h$ is proper.

Take an arbitrary $(x^*,x^{**}) \in X^* \times X^{**}$. Then we have
$$h^*(x^*,x^{**}) = (h_S^*\bigcirc_2^A h_T^*)(x^*,x^{**})$$
$$=\inf\{h_S^*(x^*,u^{**})+h_T^*(v^*,A^{**}w^{**}): u^{**},w^{**}\in X^{**}, v^*\in Y^*, u^{**}+w^{**}=x^{**},\, A^*v^*=x^*\}$$
$$\geq \inf\{\<u^{**},x^*\>+\<w^{**},x^*\>: u^{**},w^{**}\in X^{**}, u^{**}+w^{**}=x^{**}\}=\<x^{**},x^*\>.$$

Thus, according to Theorem \ref{max-cr-nref} and Theorem \ref{th-gos}, the operator with the graph
$$\{(x,x^*)\in X\times X^*:h(x,x^*)=\langle x^*,x\rangle \}$$
is maximal monotone of Gossez type (D) and one has
$$\{(x,x^*)\in X\times X^*:h(x,x^*)=\langle x^*,x\rangle \} = \{(x,x^*)\in X\times X^*:h^*(x^*,x)=\langle x^*,x\rangle \}.$$
In order to conclude the proof, we show that
$$G(\overline{S}||_A \overline{T}) = \{(x,x^*)\in X\times X^*:h^*(x^*,x)=\langle x^*,x\rangle \}$$
and this will mean that $h$ is a strong representative function of $\overline{S}||_A \overline{T}$.

Let $(x,x^*)\in G(\overline{S}||_A \overline{T}).$ Then $x\in \overline S^{-1}(x^*)+(A^*\overline TA^{**})^{-1}(x^*),$ hence there exists $u^{**}\in \overline S^{-1}(x^*)$ and $w^{**}\in (A^*\overline TA^{**})^{-1}(x^*)$ such that $x=u^{**}+w^{**}$. Thus $(u^{**},x^*) \in G(\overline S)$ and, as $x^*\in A^*\overline TA^{**}(w^{**})$,  there exists $v^*\in \overline T(A^{**}w^{**})$ such that $A^*v^*=x^*$. Consequently, $h_S^*(x^*,u^{**}) = \langle u^{**}, x^*\rangle$ and $h_T^*(v^*,A^{**}w^{**}) =  \langle A^{**}w^{**}, v^*\rangle$ and, so,
$$h^*(x^*,x)= (h_S^*\bigcirc_2^A h_T^*)(x^*,x) \le h_S^*(x^*,u^{**})+h_T^*(v^*,A^{**}w^{**}) = \langle u^{**}, x^*\rangle + \langle w^{**}, x^*\rangle = \langle x^*,x\rangle.$$
On the other hand, as shown above, $h^*(x^*,x)\ge \<x^*,x\>$ for all $(x,x^*)\in X\times X^*,$ hence $h^*(x^*,x)=\<x^*,x\>$, implying that $G(\overline{S}||_A \overline{T}) \subseteq\{(x,x^*) \in X \times X^*:h(x,x^*)=\<x^*,x\>\}$.

Conversely, let be $(x,x^*) \in X \times X^*$ such that $h^*(x^*,x)=\<x^*,x\>$. Using that $h_S^*\bigcirc_2^A h_T^*$ is exact at $(x^*,x)$, there exists $(\overline u^{**}, \overline w^{**}, \overline v^*)\in X^{**}\times X^{**}\times Y^*$ such that $\overline u^{**}+\overline w^{**}=x, A^*\overline{v}^*=x^*$ and $\<x^*,x\> = h(x,x^*) = h_S^*(x^*,\overline u^{**})+h_T^*(\overline{v}^*,A^{**}\overline w^{**})$. Since, on the other hand, $h_S^*(x^*,\overline u^{**})+h_T^*(\overline{v}^*,A^{**}\overline w^{**}) \geq \<\overline u^{**}, x^*\> +  \<A^{**}\overline w^{**}, \overline v^*\> = \<x^*,x\>$, it follows that $h_S^*(x^*,\overline u^{**}) = \<\overline u^{**}, x^*\>$ and $h_T^*(\overline{v}^*,A^{**}\overline w^{**}) = \<A^{**}\overline w^{**}, \overline v^*\>$.

But ${h_S^*}$ and ${h_T^*}$ are representative functions of $\overline S^{-1}$ and $\overline T^{-1}$, respectively, which means that $(\overline u^{**},x^*)\in G(\overline S)$ and $(A^{**}\overline w^{**},\overline{v}^*)\in G(\overline T)$. We have $\overline u^{**} \in \overline S^{-1}(x^*)$ and, since $\overline w^{**} =x-\overline u^{**}$, we obtain $\overline{v}^*\in \overline TA^{**}(x-\overline u^{**})$, hence $x^*=A^*\overline{v}^*\in A^*\overline TA^{**}(x-\overline u^{**})$ or, equivalently, $x-\overline u^{**}\in (A^*\overline TA^{**})^{-1}(x^*).$ Thus $x=\overline{u^{**}}+(x-\overline u^{**})\in (\overline S^{-1}+(A^*\overline TA^{**})^{-1})(x^*)$ and so $(x,x^*)\in G(\overline{S}||_A \overline{T})$.

Hence, $G(\overline{S}||_A \overline{T}) = \{(x,x^*) \in X \times X^*: h^*(x^*,x) = \langle x^*,x\rangle\}$ and this concludes the proof.
\end{proof}

Under the additional assumption that the domain of Gossez's closure of $S$ is a subset of $X$, the conditions (a) and (b) of the previous theorem become sufficient for the maximal monotonicity of Gossez type (D) of the generalized parallel sum $S ||_A T$. One can notice that $D(\overline S) \subseteq X$ is particulary fulfilled when $X$ is a \textit{reflexive Banach space}.

\begin{thm}\label{t3.2} Let $S:X\rightrightarrows X^*$ and $T:Y\rightrightarrows Y^*$ be two maximal monotone operators of Gossez type (D) with strong representative functions $h_S$ and $h_T$, respectively, and $A:X\To Y$ a linear continuous mapping such that $\dom h_T \times \pr_{X^*}(\dom h_S) \cap \im A \times\Delta_{Y^*}^{A^*}\neq\emptyset$ and $D(\overline S) \subseteq X$. Assume that one of the following conditions is fulfilled:
\begin{enumerate}
\item[(a)] $(0,0,0)\in {^{\ic}( \dom h_T \times \pr_{X^*}(\dom h_S)-\im A \times\Delta_{Y^*}^{A^*})}$;
\item[(b)] the set $\{(u^*,A^*v^*,A^{**}u^{**}+v^{**},r): r \in \Real, h_S^*(u^*,u^{**})+h_T^*(v^*,v^{**})\le r\}$ is closed regarding $\Delta_{X^*}\times \im A^{**}\times \Real$ in  $(X^*,w^*)\times (X^*,w^*)\times (Y^{**},w^*)\times \Real$.
\end{enumerate}
Then the function $h:X\times X^*\To\oReal,\, h(x,x^*)= \cl\nolimits_{\|\cdot\|\times\|\cdot\|_*}(h_S\bigcirc_1^A h_T)(x,x^*)$, is a strong representative function of $S ||_A T$ and the generalized parallel sum $S ||_A T$ is a maximal monotone operator of Gossez type (D).
\end{thm}

\begin{proof} We need only to show that $\overline S ||_A \overline T=S ||_A T$, whenever $D(\overline S) \subseteq X$. Indeed,
$(x,x^*) \in G(\overline S ||_A \overline T)$ if and only if there exist $u^{**}\in \overline S^{-1}(x^*) \subseteq X$ and $w^{**}\in (A^*\overline TA^{**})^{-1}(x^*)$ such that $x=u^{**}+w^{**}$. This is further equivalent to the existence of $u^{**}$ and $w^{**}$ in $X$ such that $(u^{**},x^*) \in G(S)$, $x^*\in A^*\overline TA^{**}(w^{**}) =  A^*\overline T(Aw^{**}) =  A^*T(Aw^{**})$ and $x=u^{**}+w^{**}$. But this is the same with $x \in S^{-1}(x^*) + (A^*TA)^{-1}(x^*)$ or, equivalently, $(x,x^*) \in G(S ||_AT)$.
\end{proof}

\begin{rmk}\label{relcond}\rm
Concerning the two sufficient conditions for maximal monotonicity considered in Theorem \ref{t3.1} and Theorem \ref{t3.2}, one can notice, according to Theorem \ref{t2.5}, that condition (b) is fulfilled whenever condition (a) is fulfilled. In the last section of the paper we provide a situation where the latter fails, while condition (b) is valid (see Example \ref{ex5.1}).
\end{rmk}

In the last part of this section we turn our attention to the formulation of further interiority-type regularity conditions for the maximal monotonicity of Gossez type (D) of the generalized parallel sums $\overline{S}||_A\overline{ T}$, respectively, $S||_A T$, this time expressed by means of the \textit{graph} of $T$ and of the \textit{range} of $S$. We start with the following result.

\begin{thm}\label{t4.1} Let $S:X\rightrightarrows X^*$ and $T:Y\rightrightarrows Y^*$ be two maximal monotone operators of Gossez type (D) with strong representative functions $h_S$ and $h_T$, respectively, and $A:X\To Y$ a linear continuous mapping such that $\dom h_T \times \pr_{X^*}(\dom h_S) \cap \im A \times\Delta_{Y^*}^{A^*}\neq\emptyset$. Then it holds:
$${^{\ic}\left(G(T)\times R(S)-\im A \times \Delta_{Y^*}^{A^*}\right)} \subseteq {^{\ic}\left(\co G(T)\times \co R(S)-\im A \times \Delta_{Y^*}^{A^*}\right)} \subseteq $$
$${^{\ic}\left(\dom h_T \times \pr\nolimits_{X^*}(\dom h_S)-\im A\times \Delta_{Y^*}^{A^*}\right)} = \ri\left(\dom h_T \times \pr\nolimits_{X^*}(\dom h_S)-\im A\times \Delta_{Y^*}^{A^*}\right).$$
\end{thm}
\begin{proof}

Let us denote by $C:=\dom h_T \times \pr_{X^*}(\dom h_S) - \im A\times \Delta_{Y^*}^{A^*}$ and by $D:=G(T)\times R(S)-\im A\times \Delta_{Y^*}^{A^*}$. Then $\co D =  \co G(T)\times \co R(S)-\im A \times \Delta_{Y^*}^{A^*}$ and, obviously, ${^{\ic} D} \subseteq {^{\ic} (\co D)}$. On the other hand, as pointed out in Remark \ref{r2.4}, we have ${^{\ic} C} = \ri C$. Thus, it remains to show that ${^{\ic} (\co D)}\subseteq {^{\ic} C}$.

Since, $\co D \subseteq C$, one has $\aff (\co D) = \aff D \subseteq \aff C$. Thus, in order to prove that ${^{\ic} (\co D)}\subseteq {^{\ic} C} $, it is enough to show that that $\aff C \subseteq \cl(\aff D)$. The proof will rely on \cite[Lemma 20.4(b)]{Sim} (for another result, where this lemma found application we refer to \cite{LYao}). What we will actually prove, is that
\begin{equation}\label{int4.1}
\dom \varphi_T\times \pr\nolimits_{X^*}(\dom \varphi_S)\subseteq \cl(\aff D),
\end{equation}
where $\varphi_S$ and $\varphi_T$ denote the Fitzpatrick functions of the operators $S$ and $T$, respectively. If \eqref{int4.1} is true, then one gets
$$C \subseteq \dom \varphi_T\times \pr\nolimits_{X^*}(\dom \varphi_S)  - \im  A\times\Delta_{Y^*}^{A^*} \subseteq \cl\left(\aff D - \im  A\times\Delta_{Y^*}^{A^*}\right) = \cl(\aff D),$$
which leads to the desired conclusion.

In order to show \eqref{int4.1}, we assume without loss of generality that $(0,0)\in G(S)$ and $(0,0)\in G(T)$. Suppose that there exists $(v,v^*,u^*)\in \dom \varphi_T \times \pr_{X^*}(\dom \varphi_S )$ such that $(v,v^*,u^*)\not\in\cl(\aff D)$. Then, according to a strong separation theorem, there exist $\d\in\Real$ and $(q^*,q^{**},p^{**})\in Y^*\times Y^{**}\times X^{**}$ such that $$\<(q^*,q^{**},p^{**}),(v,v^*,u^*)\>> \d> \sup\{\<(q^*,q^{**},p^{**}),(y,y^*,x^*)\>:(y,y^*,x^*)\in \cl(\aff D)\}.$$

As $0 \in D$, $\aff D$ is a linear subspace. Thus $\<(q^*,q^{**},p^{**}),(y,y^*,x^*)\>=0$ for all $(y,y^*,x^*)\in \aff D$ and, consequently, $\d > 0$. In other words,
\begin{equation}\label{int24.1}
\<(q^*,q^{**},p^{**}),(y-Au,y^*-v^*,x^*-A^*v^*)\>=0 \ \forall (y,y^*)\in G(T) \ \forall x^*\in R(S) \ \forall u\in X \ \forall v^*\in Y^*.
\end{equation}
By taking $(y,y^*,x^*):=(0,0,0) \in G(T) \times R(S)$, we obtain
$$q^{**}=-A^{**}p^{**}$$
and from here it results that
$$\<q^*,Au\>=\<A^*q^*,u\>=0 \ \forall u\in X,$$
which means that $A^*q^*=0$. Hence,
$$\<q^{**},q^*\>= \<-A^{**}p^{**},q^*\>= \<-p^{**},A^*q^*\>=0.$$
On the other hand, from \eqref{int24.1}, we have  $\<(q^*,q^{**},p^{**}),(y,y^*,x^*)\>=0$ for all $(y,y^*)\in G(T)$ and all $x^*\in R(S)$, hence
$$\<(q^*,q^{**}),(y,y^*)\>=0 \ \forall (y,y^*)\in G(T)$$
and
$$\<p^{**}, x^*\>=0 \ \forall x^* \in R(S).$$

Take now an arbitrary $(y^{**},y^*)\in G(\overline{T}).$ Then there exists $(y_\a,y^*_\a)_{\a\in\mathfrak{I}} \in G(T)$ such that $(y_\a)_{\a\in\mathfrak{I}}$ converges to $y^{**}$ in the weak$^*$ topology of $Y^{**}$ and $(y^*_\a)_{\a\in\mathfrak{I}}$ converges to $y^*$ in the strong topology of $Y^*$. Since $(y_\a,y^*_\a)\in G(T),$ we have $\<(q^*,q^{**}),(y_\a,y_\a^*)\>=0$ for every $\a\in\mathfrak{I}$, hence $\<(q^*,q^{**}),(y^{**},y^*)\>=0$. Consequently,
$$\<(q^*,q^{**}),(y^{**},y^*)\>=0 \ \forall (y^{**},y^*)\in G(\overline{T})$$ and one can prove in a similar way that
$$\<p^{**}, x^*\>=0 \ \forall x^* \in R(\overline{S}).$$

From here, according to \cite[Lemma 20.4(b)]{Sim}, one has (as $\<q^{**},q^*\>=0$)
$$\<(q^*,q^{**}),(y^{**},y^*)\>=0 \ \forall (y^{**},y^*)\in \dom \varphi_{\overline{T}}$$ and (as, obviously, $\<p^{**},0\>=0$)
$$\<(0,p^{**}), (x^{**},x^*)\>=0 \ \forall (x^*,x^{**}) \in \dom \varphi_{\overline{S}}.$$
But $(v,v^*,u^*)\in \dom \varphi_T \times \pr_{X^*}(\dom \varphi_S )$ and, as $\varphi_{\overline{S}}|_{X \times X^*} = \varphi_S$ and $\varphi_{\overline{T}}|_{Y \times Y^*} = \varphi_T$, it follows that
$$\<(q^*,q^{**},p^{**}),(v,v^*,u^*)\>=0,$$
which is a contradiction to $\d > 0$. Consequently, \eqref{int4.1} is valid and, so, ${^{\ic}(\co D)} \subseteq {^{\ic}(C)}.$
\end{proof}

The above theorem gives rise to two supplementary interiority-type regularity conditions for the maximal monotonicity of $\overline{S}||_A\overline{T}$.

\begin{cor}\label{c4.1} Let $S:X\rightrightarrows X^*$ and $T:Y\rightrightarrows Y^*$ be two maximal monotone operators of Gossez type (D) with strong representative functions $h_S$ and $h_T$, respectively, and $A:X\To Y$ a linear continuous mapping such that $\dom h_T \times \pr_{X^*}(\dom h_S) \cap \im A \times\Delta_{Y^*}^{A^*}\neq\emptyset$. If $$(0,0,0)\in {^{\ic}\left(G(T)\times R(S)-\im A\times \Delta_{Y^*}^{A^*}\right)}$$
or
$$(0,0,0)\in {^{\ic}\left(\co G(T) \times \co R(S) - \im A\times \Delta_{Y^*}^{A^*}\right)},$$
then the extended generalized parallel sum $\overline{S}||_A \overline{T}$ is a maximal monotone operator of Gossez type (D).
\end{cor}

As follows from the following result, under the supplementary assumption that $D(\overline S) \subseteq X$, the inclusion relations in Theorem \ref{t4.1} become equalities.

\begin{thm}\label{t4.2}  Let $S:X\rightrightarrows X^*$ and $T:Y\rightrightarrows Y^*$ be two maximal monotone operators of Gossez type (D) with strong representative functions $h_S$ and $h_T$, respectively, and $A:X\To Y$ a linear continuous mapping such that $\dom h_T \times \pr_{X^*}(\dom h_S) \cap \im A \times\Delta_{Y^*}^{A^*}\neq\emptyset$ and $D(\overline S) \subseteq X$. Then it holds:
$${^{\ic}\left(G(T)\times R(S)-\im A \times \Delta_{Y^*}^{A^*}\right)}  = \ri\left(G(T)\times R(S)-\im A \times \Delta_{Y^*}^{A^*}\right) =$$
$${^{\ic}\left(\co G(T)\times \co R(S)-\im A \times \Delta_{Y^*}^{A^*}\right)} = \ri \left(\co G(T)\times \co R(S)-\im A \times \Delta_{Y^*}^{A^*}\right)=$$
$${^{\ic}\left(\dom h_T \times \pr\nolimits_{X^*}(\dom h_S)-\im A\times \Delta_{Y^*}^{A^*}\right)} = \ri\left(\dom h_T \times \pr\nolimits_{X^*}(\dom h_S)-\im A\times \Delta_{Y^*}^{A^*}\right).$$
\end{thm}
\begin{proof} By keeping the notations introduced in the proof of Theorem \ref{t4.1}, let us prove first that ${^{\ic} C} \subseteq D$. Take an arbitrary $(v,v^*,u^*)\in {^{\ic} C}$, hence $(0,0,0)\in {^{\ic} (C-(v,v^*,u^*))}$. Consider the functions $$\tilde f: X\times X^*\To\oReal,\,\tilde{f}(x,x^*)=h_S(x,x^*+u^*)-\<u^*,x\>$$ and
$$\tilde{g}:Y\times Y^*\To\oReal,\, \tilde{g}(y,y^*)=h_T(y+v,y^*+v^*)-(\<v^*,y\>+\<y^*,v\>+\<v^*,v\>)$$
and the operators $\widetilde{S}:X\rightrightarrows X^*$ defined by $G(\widetilde{S})=\{(x,x^*)\in X\times X^*:\tilde{f}(x,x^*)=\<x^*,x\>\}$ and $\widetilde{T}:Y\rightrightarrows Y^*$ defined by $G(\widetilde{T})=\{(y,y^*)\in Y\times Y^*:\tilde{g}(y,y^*)=\<y^*,y\>\}.$ It can be easily observed, that $G(\widetilde{S})=G(S)-(0,u^*)$ and $G(\widetilde{T})=G(T)-(v,v^*)$. Consequently, $\widetilde{S}$ and $\widetilde{T}$ are maximal monotone operators of Gossez type (D) and $\tilde{f},$ respectively, $\tilde{g}$ are strong representative functions for them. Since $D(\overline S) \subseteq X$, the domain of Gossez's closure of $\widetilde{S}$ is a subset of $X$, too. Hence, according to Theorem \ref{t3.2}, the condition
$$(0,0,0)\in {^{\ic}(C-(v,v^*,u^*))} = {^{\ic}\left(\dom \tilde{g} \times \pr\nolimits_{X^*}(\dom \tilde{f} )-\im A\times \Delta_{Y^*}^{A^*}\right)}$$
ensures the maximal monotonicity of $\widetilde{S}||_A \widetilde{T}$. Hence, $G(\widetilde{S}||_A \widetilde{T})\neq\emptyset,$ thus there exists $x^*\in (\widetilde{S}^{-1}+(A^*\widetilde{T}A)^{-1})^{-1}(x)$ for some $x\in X.$ This means that there exist $u, w \in X$ such that $(u,x^*)\in G(\widetilde{S})$ and $(w,x^*)\in G(A^*\widetilde{T}A)$ and $u+w=x$. As $G(\widetilde{S})=G(S)-(0,u^*)$, we have
$$(0,u^*)\in G(S)-(u,x^*).$$
On the other hand, as $x^*\in A^*\widetilde{T}A(w),$ there exists $y^*\in Y^*,$ such that $y^*\in\widetilde{T}(Aw)$ and $x^*=A^*y^*$. Thus, for $y:=Aw$, we have $(y,y^*)\in G(\widetilde{T})=G(T)-(v,v^*),$ hence
$$(v,v^*)\in G(T)-(y,y^*).$$
In conclusion, $(v,v^*,u^*)\in G(T)\times R(S)-\im A\times \Delta_{Y^*}^{A^*}=D$ and, so, ${^{\ic} C} \subseteq D$.

If ${^{\ic} C}=\ri C$ is empty, then by Theorem \ref{t4.1} it holds ${^{\ic} D} = {^{\ic}(\co D)} = {^{\ic} C} = \ri C = \emptyset$. Consequently, $\ri D = \ri(\co D) = \emptyset$.

Assume now that ${^{\ic} C}$ is nonempty. Since ${^{\ic} C} \subseteq D \subseteq \co D \subseteq C$, one gets that ${^{\ic} D} = {^{\ic}(\co D)} = {^{\ic} C} = \ri C$. Moreover, it holds $\aff(^{\ic} C) = \aff C$ and, as $\ri C = ^{\ic} C \subseteq D \subseteq \co D \subseteq C$, we have $\aff C = \aff D$, these sets being closed. Thus $\ri C = \ri D = \ri(\co D)$ and this provides the desired conclusion.
\end{proof}

We close the section by the following characterization of the maximal monotonicity of Gossez type (D) of $S||_A T$, which follows from Theorem \ref{t3.2} and Theorem \ref{t4.2}.

\begin{cor}\label{c4.2} Let $S:X\rightrightarrows X^*$ and $T:Y\rightrightarrows Y^*$ be two maximal monotone operators of Gossez type (D) with strong representative functions $h_S$ and $h_T$, respectively, and $A:X\To Y$ a linear continuous mapping such that $\dom h_T \times \pr_{X^*}(\dom h_S) \cap \im A \times\Delta_{Y^*}^{A^*}\neq\emptyset$ and $D(\overline S) \subseteq X$. Then one has the following sequence of equivalencies
$$(0,0,0) \in {^{\ic}\left(G(T)\times R(S)-\im A \times \Delta_{Y^*}^{A^*}\right)}  \Leftrightarrow (0,0,0) \in \ri\left(G(T)\times R(S)-\im A \times \Delta_{Y^*}^{A^*}\right) \Leftrightarrow $$
$$(0,0,0) \in {^{\ic}\left(\co G(T)\times \co R(S)-\im A \times \Delta_{Y^*}^{A^*}\right)} \Leftrightarrow (0,0,0) \in \ri \left(\co G(T)\times \co R(S)-\im A \times \Delta_{Y^*}^{A^*}\right) \Leftrightarrow$$
$$(0,0,0) \in {^{\ic}\left(\dom h_T \times \pr\nolimits_{X^*}(\dom h_S)-\im A\times \Delta_{Y^*}^{A^*}\right)} \Leftrightarrow (0,0,0) \in \ri\left(\dom h_T \times \pr\nolimits_{X^*}(\dom h_S)-\im A\times \Delta_{Y^*}^{A^*}\right)$$
and each of these conditions guarantees that the  generalized parallel sum $S||_A T$ is a maximal monotone operator of Gossez type (D).
\end{cor}

\section{Particular cases}

In this section we will consider two particular instances of the generalized parallel sum defined via a linear continuous mapping and show what the results provided in Section 3 become in these special settings.

\subsection{The maximal monotonicity of Gossez type (D) of $S||T$}

Assume that $X$ is a real nonzero Banach space and $S:X\rightrightarrows X^*$ and $T:X\rightrightarrows X^*$ are two monotone operators.  By taking $A=\id_X:X \rightarrow X$, their extended generalized parallel sum defined via $A$ and their generalized parallel sum defined via $A$ become the \textit{extended parallel sum of $S$ and $T$}
$$\overline{S}|| \overline{T}:X\rightrightarrows X^*,\overline{S}||\overline{T}(x):=(\overline{S}^{-1} + \overline{T}^{-1})^{-1}(x) \ \forall x \in X$$
and the \textit{classical parallel sum} of $S$ and $T$,
$$S||T:X\rightrightarrows {X^*},\,S|| T(x):= (S^{-1}+ T^{-1})^{-1}(x )\ \forall x \in X,$$
respectively.

Having $h_S:X\times X^*\To\overline{\Real}$ and $h_T:X\times X^*\To\overline{\Real}$ representative functions of $S$ and $T$, respectively, the extended infimal convolutions of them, namely $h_S\bigcirc_1^A h_T$ and $h_S^*\bigcirc_2^A h_T^*$, turn out to be the following classical bivariate infimal convolutions (see, for instance, \cite{Bo-Cse, voiseizal, Sim, simons-zal})
$$h_S \square_1 h_T : X \times X^* \rightarrow \oReal, (h_S\square_1 h_T)(x,x^*)=\inf \{h_S(u,x^*)+h_T(w,x^*):\, u,w \in X, u+w=x\}$$
and
$$h_S^*\square_2 h_T^* : X^* \times X^{**} \rightarrow \oReal, (h_S^*\square_2 h_T^*)(x^*,x^{**})=\inf \{h_S^*(x^*,u^{**})+h_T^*(x^*,w^{**}):\, u^{**},w^{**}\in X^{**}, u^{**}+w^{**}=x^{**}\},$$
respectively.

\begin{thm}\label{t5.1} Let $S:X\rightrightarrows X^*$ and $T:X\rightrightarrows X^*$ be two maximal monotone operators of Gossez type (D) with strong representative functions $h_S$ and $h_T$, respectively, such that $\pr_{X^*}(\dom h_S)\cap \pr_{X^*}(\dom h_T) \neq\emptyset$ and assume that one of the following conditions is fulfilled:
\begin{enumerate}
\item[(a)] $0 \in {^{\ic}(\pr_{X^*}(\dom h_S)-\pr_{X^*}(\dom h_T))}$;
\item[(b)] the set $\{(u^*,v^*,u^{**}+v^{**},r): r \in \Real, h_S^*(u^*,u^{**})+h_T^*(v^*,v^{**})\le r\}$ is closed regarding $\Delta_{X^*}\times X^{**}\times \Real$ in  $(X^*,w^*)\times (X^*,w^*)\times (X^{**},w^*)\times \Real$.
\end{enumerate}
Then the following statements are true:
\begin{enumerate}

\item[(i)] The function $h:X\times X^*\To\oReal,\, h(x,x^*) = \cl\nolimits_{\|\cdot\|\times\|\cdot\|_*}(h_S\square_1 h_T)(x,x^*)$, is a strong representative function of $\overline{S}||\overline{T}$ and the extended  parallel sum $\overline{S}|| \overline{T}$ is a maximal monotone operator of Gossez type (D).

\item[(ii)] If $D(\overline S) \subseteq X$ (or, if $D(\overline T) \subseteq X$), then the function $h:X\times X^*\To\oReal,\, h(x,x^*) = \cl\nolimits_{\|\cdot\|\times\|\cdot\|_*}(h_S\square_1 h_T)(x,x^*)$, is a strong representative function of $S || T$ and the parallel sum $S || T$ is a maximal monotone operator of Gossez type (D).
\end{enumerate}
\end{thm}

\begin{proof}
The result follows directly Theorem \ref{t4.1} and Theorem \ref{t4.2}, by noticing that the interiority-type condition in these two statements becomes
$$(0,0,0)\in {^{\ic} (\dom h_T \times \pr\nolimits_{X^*}(\dom h_S)- X \times\Delta_{X^*})}  = X \times {^{\ic} (\pr\nolimits_{X^*}(\dom h_S) \times \pr\nolimits_{X^*}(\dom h_S)- \Delta_{X^*})}$$
or, equivalently,
$$(0,0)\in {^{\ic} (\pr\nolimits_{X^*}(\dom h_S) \times \pr\nolimits_{X^*}(\dom h_S)- \Delta_{X^*})}.$$
According to Lemma \ref{l2.1}, the latter relation is equivalent to
$$0 \in {^{\ic}(\pr\nolimits_{X^*}(\dom h_S)-\pr\nolimits_{X^*}(\dom h_T))}.$$
\end{proof}

The next result follows from Theorem \ref{t4.1} and Theorem \ref{t5.1}(i).

\begin{thm}\label{t5.2} Let $S:X\rightrightarrows X^*$ and $T:X\rightrightarrows X^*$ be two maximal monotone operators of Gossez type (D) with strong representative functions $h_S$ and $h_T$, respectively, such that $\pr_{X^*}(\dom h_S)\cap \pr_{X^*}(\dom h_T) \neq\emptyset$.
\begin{enumerate}
\item [(a)] Then it holds:
$${^{\ic}\left(R(S)-R(T)\right)} \subseteq {^{\ic}\left(\co R(S)-\co R(T)\right)} \subseteq$$
$${^{\ic}(\pr\nolimits_{X^*}(\dom h_S)-\pr\nolimits_{X^*}(\dom h_T))} = \ri(\pr\nolimits_{X^*}(\dom h_S)-\pr\nolimits_{X^*}(\dom h_T)).$$

\item [(b)] If $$0 \in {^{\ic}\left(R(S)-R(T)\right)}$$
or
$$0 \in {^{\ic}\left(\co R(S)-\co R(T)\right)},$$
then the extended parallel sum $\overline{S}|| \overline{T}$ is a maximal monotone operator of Gossez type (D).
\end{enumerate}
\end{thm}

\begin{proof}
As (b) is a direct consequence of Theorem \ref{t5.1}(i) and statement (a), we will turn our attention to the proof of the latter. Concerning it, one can easily notice that the inclusion
$${^{\ic}\left(R(S)-R(T)\right)} \subseteq {^{\ic}\left(\co R(S)-\co R(T)\right)}$$
follows directly from the definition of the intrinsic relative algebraic interior, while the equality
$${^{\ic}(\pr\nolimits_{X^*}(\dom h_S)-\pr\nolimits_{X^*}(\dom h_T))} = \ri(\pr\nolimits_{X^*}(\dom h_S)-\pr\nolimits_{X^*}(\dom h_T))$$
is a direct consequence of \cite[Theorem 2.7.2]{Zal-carte}, applied to the proper, convex and lower semicontinuous function
$$\Phi : X \times X \times X^* \times X^* \rightarrow \oReal, \Phi(x,u,x^*,u^*) = h_S(x,x^*+u^*) + h_T(u,x^*),$$
by taking into account that (we consider the projection on the fourth component of the product space $ X \times X \times X^* \times X^*$)
$$\pr\nolimits_{X^*} (\dom \Phi) = \pr\nolimits_{X^*}(\dom h_S)-\pr\nolimits_{X^*}(\dom h_T).$$

What it remained to be shown, namely that
$${^{\ic}\left(\co R(S)-\co R(T)\right)} \subseteq {^{\ic}(\pr\nolimits_{X^*}(\dom h_S)-\pr\nolimits_{X^*}(\dom h_T))},$$
follows according to Lemma \ref{l2.1} and Theorem \ref{t4.1}. Indeed, when $u^* \in {^{\ic}\left(\co R(S)-\co R(T)\right)}$ or, equivalently, $0 \in {^{\ic}\left(\co R(S)-u^* - \co R(T)\right)}$, one has that
$$(u^*,0) \in {^{\ic}(\co R(T) \times \co R(S) -\Delta_{X^*})} \subseteq {^{\ic}(\pr\nolimits_{X^*}(\dom h_T) \times\pr\nolimits_{X^*}(\dom h_S)-\Delta_{X^*})}$$
and from here, again via Lemma \ref{l2.1}, it follows $u^* \in {^{\ic}(\pr\nolimits_{X^*}(\dom h_S)-\pr\nolimits_{X^*}(\dom h_T))}$.
\end{proof}

Theorem \ref{t4.2} and Theorem \ref{t5.1}(ii) give rise to the following result.

\begin{thm}\label{t5.3} Let $S:X\rightrightarrows X^*$ and $T:X\rightrightarrows X^*$ be two maximal monotone operators of Gossez type (D) with strong representative functions $h_S$ and $h_T$, respectively, such that $\pr_{X^*}(\dom h_S)\cap \pr_{X^*}(\dom h_T) \neq\emptyset$ and $D(\overline S) \subseteq X$ (or, $D(\overline T) \subseteq X$).
\begin{enumerate}
\item [(a)] Then it holds:
$${^{\ic}\left(R(S)-R(T)\right)}  = \ri \left(R(S)-R(T)\right) = {^{\ic}\left(\co R(S)-\co R(T)\right)} = \ri \left(\co R(S)-\co R(T)\right)=$$
$${^{\ic}(\pr\nolimits_{X^*}(\dom h_S)-\pr\nolimits_{X^*}(\dom h_T))} = \ri(\pr\nolimits_{X^*}(\dom h_S)-\pr\nolimits_{X^*}(\dom h_T)).$$
\item [(b)] One has the following sequence of equivalencies
$$0 \in {^{\ic}\left(R(S)-R(T)\right)}  \Leftrightarrow 0 \in \ri \left(R(S)-R(T)\right) \Leftrightarrow 0 \in {^{\ic}\left(\co R(S)-\co R(T)\right)} \Leftrightarrow$$ $$0 \in \ri \left(\co R(S)-\co R(T)\right) \Leftrightarrow 0 \in {^{\ic}(\pr\nolimits_{X^*}(\dom h_S)-\pr\nolimits_{X^*}(\dom h_T))} \Leftrightarrow 0 \in \ri(\pr\nolimits_{X^*}(\dom h_S)-\pr\nolimits_{X^*}(\dom h_T))$$
and each of these conditions guarantees that the parallel sum $S||T$ is a maximal monotone operator of Gossez type (D).
\end{enumerate}
\end{thm}

\begin{proof}
We will only prove statement (a), as (b) is a direct consequence of it and Theorem \ref{t5.1}(ii).

For an arbitrary $u^* \in {^{\ic}(\pr\nolimits_{X^*}(\dom h_S)-\pr\nolimits_{X^*}(\dom h_T))}$ one has, via Lemma \ref{l2.1}, that
$$(u^*,0) \in  {^{\ic}(\pr\nolimits_{X^*}(\dom h_T) \times\pr\nolimits_{X^*}(\dom h_S)-\Delta_{X^*})}.$$  Further, by Theorem \ref{t4.2} it follows $(u^*,0) \in (R(T) \times R(S)-\Delta_{X^*})$,
implying that $u^* \in R(S)-R(T)$. Consequently,
$${^{\ic}(\pr\nolimits_{X^*}(\dom h_S)-\pr\nolimits_{X^*}(\dom h_T))} \subseteq R(S)-R(T).$$
If ${^{\ic}(\pr\nolimits_{X^*}(\dom h_S)-\pr\nolimits_{X^*}(\dom h_T))}$ is empty, then there is nothing to be proved. Otherwise, the conclusion follows, by using that $${^{\ic}(\pr\nolimits_{X^*}(\dom h_S)-\pr\nolimits_{X^*}(\dom h_T))} \subseteq R(S)-R(T) \subseteq \co R(S) - \co R(T) \subseteq \pr\nolimits_{X^*}(\dom h_S)-\pr\nolimits_{X^*}(\dom h_T)$$ and $\aff ({^{\ic}(\pr\nolimits_{X^*}(\dom h_S)-\pr\nolimits_{X^*}(\dom h_T))}) = \aff (\pr\nolimits_{X^*}(\dom h_S)-\pr\nolimits_{X^*}(\dom h_T))$.
\end{proof}

\begin{rmk}\label{r5.2}\rm In the setting of reflexive Banach spaces several interority-type regularity conditions ensuring the maximal monotonicity of the parallel sum $S||T$ of two maximal monotone operators $S$ and $T$ have been introduced in the literature. While in \cite{At-Ch-Mou}  the condition
\begin{equation*}\label{e1}
\inte(R(S))\cap R(T)\neq\emptyset
\end{equation*}
was considered, in \cite{Ri} it has been assumed that
\begin{equation*}\label{e4}
\cone(R(S)- R(T))=X^*.
\end{equation*}
Further, in a Hilbert space context, in \cite{Mou} the condition
\begin{equation*}\label{e2}
\cone(R(S)- R(T)) \ \mbox{is a closed linear subspace of} \ X^*
\end{equation*}
has been stated, while in \cite{Pen-Zal}, in reflexive Banach spaces, the condition
\begin{equation*}\label{e3}
\cone(\co R(S)- \co R(T)) \ \mbox{is a closed linear subspace of} \ X^*
\end{equation*}
was proposed.

Taking into account that an operator $T:X \rightrightarrows X^*$ is maximal monotone if and only if $T^{-1}:X^* \rightrightarrows X$ is maximal monotone and that $D(T^{-1})=R(T)$, one can easily observe that all these interiority-type regularity conditions ensuring that $S||T$ is maximal monotone, provided $S$ and $T$ are maximal monotone, are the counterpart of some meanwhile classical ones stated for the maximal monotonicity of the sum $S^{-1}+T^{-1}$ (see, for instance, \cite{Sim, Pen-Zal1, simons-zal}) and can be easily derived from them.

For interiority-type regularity conditions guaranteeing the maximal monotonicity of Gossez type (D) of the parallel sum and the extended parallel sum of two maximal monotone operators of Gossez type (D) in general Banach spaces we refer to \cite{Sim3}. These results have been obtained as particular instances of some corresponding ones formulated for the generalized parallel sum defined via a linear continuous mapping $S||^A T$.
\end{rmk}

\begin{ex}\label{ex5.1}\rm With this example we want to emphasize that there exist maximal monotone operators with a maximal monotone parallel sum and for which the interiority-type regularity condition (a) in Theorem \ref{t5.1} is not fulfilled, while the closedness-type condition (b) in Theorem \ref{t5.1} holds.

Consider the proper, sublinear and lower semicontinuous functions $f,g:\Real^2 \To \oReal$, $f(x_1,x_2)=\|(x_1,x_2)\|_2+\delta_{\Real^2_+}(x_1,x_2),$ where $\|\cdot\|_2$ denotes the Euclidean norm on $\Real^2$, and $g(x_1,x_2) = \sqrt{3}/2 x_1+ 1/2 x_2+\delta_{-\Real_+^2}(x_1,x_2)$. Then the multivalued operators $S:=\partial f$ and $T:=\partial g$ are maximal monotone and their \textit{only} representative functions are $h_S((x_1,x_2),(x_1^*,x_2^*))=f(x_1,x_2)+f^*(x_1^*,x_2^*)$ and $h_T((x_1,x_2),(x_1^*,x_2^*))=g(x_1,x_2)+g^*(x_1^*,x_2^*)$, respectively. One can easily verify  that $f^*=\delta_{\cl B_{\Real^2}-\Real^2_+},$ where $B_{\Real^2}$ denotes the open unit ball of $\Real^2$, and $g^*=\delta_{[\sqrt{3}/{2},+\infty)\times [1/2,+\infty)}$.

Obviously,
$$\pr\nolimits_{\Real^2}(\dom h_S)\cap \pr\nolimits_{\Real^2}(\dom h_T)=(\cl B_{\Real^2}-\Real^2_+)\cap [\sqrt{3}/{2},+\infty)\times [1/2,+\infty)\neq\emptyset,$$
where the projection is taken onto the second component of the product space $\Real^2 \times \Real^2$.

We also have
$$\{(u^{*},v^{*},u^{**}+v^{**},r)\in \Real^2 \times \Real^2 \times \Real^2 \times \Real: h_S^*(u^*,u^{**})+h_T^*(v^*,v^{**})\le r\}=$$
$$\left(\cl B_{\Real^2}-\Real^2_+\right)\times [\sqrt{3}/{2},+\infty)\times [1/2,+\infty) \times \{(x_1+y_1,x_2+y_2,r)\in\Real^2\times\Real: \|(x_1,x_2)\|_2+\sqrt{3}/{2} y_1+1/2 y_2\le r\},$$ which is obviously a closed set. Hence, condition (b) in Theorem \ref{t5.1} is fulfilled and $S||T$ is maximal monotone.

On the other hand, one can notice that condition (a) in Theorem \ref{t5.1} fails. Otherwise, one would have according to Theorem \ref{t5.3}(b) that
$$(0,0) \in \ri(\pr\nolimits_{\Real^2}(\dom h_S) - \pr\nolimits_{\Real^2}(\dom h_T))$$
or, equivalently,
$$\left(B_{\Real^2} -\inte \Real^2_+\right) \cap (\sqrt{3}/{2},+\infty) \times (1/2,+\infty) \neq \emptyset,$$
which would lead to a contradiction.
\end{ex}

\subsection{The maximal monotonicity of Gossez type (D) of $A^*TA$}

For the second particular instance, we treat in this section, we stay in the same setting as in Section 3, but assume that $S:X \rightrightarrows X^*$ is the multivalued operator with $G(S) = \{0\} \times X^*$, which is obviously maximal monotone of Gossez type (D). Its extension to the bidual, $\overline S : X^{**} \rightrightarrows X^*$, fulfills $G(\overline S) = \{0\} \times X^*$, which means that the extended generalized parallel sum $\overline S||_A \overline T$ and the generalized parallel sum $S||_A T$ coincide (see also the proof of Theorem \ref{t3.2}) and fulfill
$$\overline S||_A \overline T (x) = S||_A T(x) = A^*TA(x) \ \forall x \in X.$$
Since $\varphi_S = \psi_S = \delta_{ \{0\} \times X^*}$, by Proposition \ref{frepr} it follows that the only representative function of $S$ is $h_S = \delta_{ \{0\} \times X^*}$. Since $h_S^* = \delta_{X^* \times \{0\}}$, $h_S$ is actually a strong representative function of $S$.

Having $h_T:Y\times Y^*\To\overline{\Real}$ a representative function $T$, the extended infimal convolutions $h_S\bigcirc_1^A h_T$ and $h_S^*\bigcirc_2^A h_T^*$ of $h_S$ and $h_T$ become in this situation
$$h_T^A : X \times X^* \rightarrow \oReal, h_T^A(x,x^*)=\inf \{h_T(Ax,v^*):v^* \in Y^*, A^*v^*=x^*\}$$
and
$$h_T^{*A} : X^* \times X^{**} \rightarrow \oReal, h_T^{*A}(x^*,x^{**})=\inf \{h_T^*(v^*,A^{**}x^{**}):\, v^* \in Y^*, A^*v^*=x^*\},$$
respectively.

Noticing that $\dom h_T \times \pr_{X^*}(\dom h_S) - \im A\times \Delta_{Y^*}^{A^*}=(\pr_{Y}(\dom h_T)-\im A)\times Y^*\times X^*$, Theorem \ref{t3.2} gives rise to the following result.

\begin{thm}\label{t5.4} Let $T:Y\rightrightarrows Y^*$ be a maximal monotone operators of Gossez type (D) with strong representative function $h_T$ and $A:X\To Y$ a linear continuous mapping such that $\pr_{Y}(\dom h_T) \cap \im A \neq \emptyset$. Assume that one of the following conditions is fulfilled:
\begin{enumerate}
\item[(a)] $0\in {^{\ic}(\pr_{Y}(\dom h_T)-\im A)}$;
\item[(b)] the set $\{(A^*v^*,v^{**},r): r \in \Real, h_T^*(v^*,v^{**})\le r\}$ is closed regarding $X^* \times \im A^{**}\times \Real$ in  $(X^*,w^*)\times (Y^{**},w^*)\times \Real$.
\end{enumerate}
Then the function $h:X\times X^*\To\oReal,\, h(x,x^*)= \cl\nolimits_{\|\cdot\|\times\|\cdot\|_*} h_T^A (x,x^*)$, is a strong representative function of $A^*TA$ and  $A^*TA$ is a maximal monotone operator of Gossez type (D).
\end{thm}

Since $G(T)\times R(S)-\im A\times \Delta_{Y^*}^{A^*}=(D(T)-\im A)\times Y^*\times X^*$, via Theorem \ref{t4.2} and Corollary \ref{c4.2} we obtain the following statement.

\begin{thm}\label{t5.5} Let $T:Y\rightrightarrows Y^*$ be a maximal monotone operators of Gossez type (D) with strong representative function $h_T$ and $A:X\To Y$ a linear continuous mapping such that $\pr_{Y}(\dom h_T) \cap \im A \neq \emptyset$.
\begin{enumerate}
\item [(a)] Then it holds:
$${^{\ic}\left(D(T)-\im A\right)}  = \ri \left(D(T)-\im A\right) = {^{\ic}\left(\co D(T) -\im A\right)} = \ri \left(\co D(T) -\im A\right)=$$
$${^{\ic}(\pr\nolimits_{Y}(\dom h_T)-\im A)} = \ri(\pr\nolimits_{Y}(\dom h_T)-\im A).$$
\item [(b)] One has the following sequence of equivalencies
$$0 \in {^{\ic}\left(D(T)-\im A\right)}  \Leftrightarrow 0 \in \ri \left(D(T)-\im A\right) \Leftrightarrow 0 \in {^{\ic}\left(\co D(T) -\im A\right)} \Leftrightarrow$$ $$0 \in \ri \left(\co D(T) -\im A\right) \Leftrightarrow 0 \in {^{\ic}(\pr\nolimits_{Y}(\dom h_T)-\im A)} \Leftrightarrow 0 \in \ri(\pr\nolimits_{Y}(\dom h_T)-\im A)$$
and each of these conditions guarantees that $A^*TA$ is a maximal monotone operator of Gossez type (D).
\end{enumerate}
\end{thm}

\begin{rmk}\label{r5.3}\rm
Using as a starting point Theorem \ref{t5.4} and Theorem \ref{t5.5} and by employing the techniques used in \cite{glr}, one can further provide interiority- and closedness-type regularity conditions for the maximal monotonicity of Gossez type (D) of the sum of two maximal monotone operators of Gossez type (D), but also for the sum of a maximal monotone operator of Gossez type (D) with the composition of another maximal monotone operator of Gossez type (D) with a linear continuous mapping (for the latter one will thereby rediscover the statements given in \cite[Theorem 16]{voiseizal}).
\end{rmk}

\noindent{\bf Acknowledgements.} The authors are thankful to E.R. Csetnek for pertinent comments and suggestions on an earlier draft of the article.

\end{document}